\documentclass[amscd,verbatim,12pt]{amsart}

\usepackage{tikz}
\usepackage{amssymb,amsmath,latexsym,graphics,youngtab,supertabular}
\usepackage{amsbsy}
\usepackage{amscd}
\usepackage{amsfonts}
\usepackage{amsthm}
\usepackage{epsfig}
\usepackage{times}
\usepackage{mathrsfs}
\usepackage{verbatim}
\usepackage{url}
\usepackage{graphicx}
\usepackage{txfonts}
\usepackage[pagebackref]{hyperref}

\usepackage[english]{babel}
\usepackage[T1]{fontenc}
\usepackage[utf8]{inputenc}

\usetikzlibrary{backgrounds,fit,decorations.pathreplacing}
\usetikzlibrary{arrows,shapes,positioning}
\usetikzlibrary{calc,through,backgrounds,chains}

\usetikzlibrary{arrows,shapes,snakes,automata,backgrounds,petri}

\usepackage[version=3]{mhchem}

\newtheorem{theorem}{Theorem}[section]
\newtheorem{lemma}[theorem]{Lemma}
\newtheorem{cor}[theorem]{Corollary}
\theoremstyle{definition}
\newtheorem{definition}[theorem]{Definition}
\newtheorem{proposition}[theorem]{Proposition}

\newcommand{\Z}{{\mathbb Z}}

\newcommand{\abs}[1]{\lvert#1\rvert}

\begin{document}

\title[Arithmetic properties for generalized cubic partitions]{Arithmetic properties for generalized cubic partitions and overpartitions modulo a prime}


\author[T. Amdeberhan]{Tewodros Amdeberhan}
\address{Department of Mathematics,
Tulane University, New Orleans, LA 70118, USA}
\email{tamdeber@tulane.edu}

\author[J. A. Sellers]{James A. Sellers}
\address{Department of Mathematics and Statistics, 
University of Minnesota Duluth, Duluth, MN 55812, USA}
\email{jsellers@d.umn.edu}

\author[A. Singh]{Ajit Singh}
\address{Department of Mathematics, University of Virginia, Charlottesville, VA 22904, USA}
\email{ajit18@iitg.ac.in}

\subjclass[2020]{05A17, 11F03, 11P83}

\date{\today}

\keywords{cubic partitions, Ramanujan congruences, modular forms}

\begin{abstract}
\noindent  
A cubic partition is an integer partition wherein the even parts can appear in two colors.  In this paper, we introduce the notion of generalized cubic partitions and prove a number
of new congruences akin to the classical Ramanujan-type. We emphasize two methods of proofs, one elementary (relying significantly on functional equations) and the other based on modular forms. We close by proving analogous results for generalized overcubic partitions.
\end{abstract}

\maketitle

\newcommand{\nn}{\nonumber}
\newcommand{\ba}{\begin{eqnarray}}
\newcommand{\ea}{\end{eqnarray}}
\newcommand{\realpart}{\mathop{\rm Re}\nolimits}
\newcommand{\imagpart}{\mathop{\rm Im}\nolimits}
\newcommand{\lcm}{\operatorname*{lcm}}
\newcommand{\vGam}{\varGamma}

\newtheorem{Definition}{\bf Definition}[section]
\newtheorem{Thm}[Definition]{\bf Theorem}
\newtheorem{Theorem}[Definition]{\bf Theorem}
\newtheorem{Example}[Definition]{\bf Example}
\newtheorem{Remark}[Definition]{\bf Remark}
\newtheorem{Lem}[Definition]{\bf Lemma}
\newtheorem{Note}[Definition]{\bf Note}
\newtheorem{Cor}[Definition]{\bf Corollary}
\newtheorem{Prop}[Definition]{\bf Proposition}
\newtheorem{Conj}[Definition]{\bf Conjecture}
\newtheorem{Problem}[Definition]{\bf Problem}
\numberwithin{equation}{section}

\section{Introduction} \label{sec-intro}
\setcounter{equation}{0}

\noindent

A partition $\lambda$ of a positive integer $n$ is a sequence of positive integers $\lambda_1 \geq \lambda_2 \geq \dots \geq \lambda_r$ such that $\sum_{i=1}^r \lambda_i = n$.  The values $\lambda_1,\lambda_2, \dots, \lambda_r$ are called the parts of $\lambda$.  We let $p(n)$ denote the number of partitions of $n$ for $n\geq 1$, and we define $p(0):=1.$  As an example, note that the partitions of $n=4$ are 
$$4, \ \ 3+1, \ \  2+2, \ \ 2+1+1, \ \ 1+1+1+1$$ 
and this means $p(4) = 5.$  

Throughout this work, we adopt the notations $(a;q)_{\infty}=\prod_{j\geq0}(1-aq^j)$, where $\abs{q}<1$, and $f_k:=(q^k;q^k)_\infty$.
As was proven by Euler, we know that 
\begin{align*} 
P(q):=\sum_{n\geq0}p(n)\,q^n&=\prod_{j\geq1}\frac1{1-q^j} = \frac{1}{f_1}.  
\end{align*}
Because the focus of this paper is on divisibility properties of certain partition functions, we remind the reader of the celebrated Ramanujan congruences for $p(n)$ \cite[p. 210, p. 230]{Rama}:  For all $n\geq 0,$
\begin{align} 
p(5n+4)&\equiv0\pmod 5,  \nonumber \\  
p(7n+5)& \equiv0\pmod 7, \textrm{\ \ and} \label{RamCong}\\
p(11n+6)& \equiv0\pmod{11}. \nonumber
\end{align}
Indeed, among Ramanujan’s discoveries, the equality
\begin{equation}
\label{RamBeautifulIdentity}
\sum_{n\geq0}p(5n+4)\,q^n=5\,\frac{(q^5;q^5)_{\infty}^5}{(q;q)_{\infty}^6}
\end{equation}
was regarded as his ``Most Beautiful Identity” by both Hardy and MacMahon (see \cite[p. xxxv]{Rama}).

Based on an identity on Ramanujan’s cubic continued fractions \cite{Andrews}, Chan \cite{Chan} introduced the notion of cubic partitions.  A \emph{cubic partition} of weight $n$ is a partition of $n$ wherein the even parts can appear in two colors.  We denote the number of such cubic partitions by $a_2(n)$ and define $a_2(0):=1$.  
For example, $a_2(3)=4$ since the list of cubic partitions of $n=3$ is: $3, \  {\color{red}2}+1,\  2+1, \ 1+1+1$. We also compute $a_2(4)=9$ because the cubic partitions of $4$ are 
$$
{\color{red}4},\ \ 4,\ \ 3+1, \ \ {\color{red}{2}} + {\color{red}{2}}, \ \ {\color{red}2} + 2,\ \ 2+2,\ \ {\color{red}2}+1+1,\ \ 2+1+1,\ \ 1+1+1+1.
$$
It is clear from the definition of cubic partitions that the generating function for $a_2(n)$ is given by 
\begin{align*} 
F_2(q):=\sum_{n\geq0}a_2(n)\,q^n&=\prod_{j\geq1}\frac1{(1-q^j)(1-q^{2j})} = \frac{1}{f_1f_2}.  
\end{align*}
Chan \cite{Chan} succeeded in proving the following elegant analogue of \eqref{RamBeautifulIdentity}:
$$\sum_{n\geq0}a_2(3n+2)\,q^n=3\,\frac{(q^3;q^3)_{\infty}^3(q^6;q^6)_{\infty}^3}{(q;q)_{\infty}^4(q^2;q^2)_{\infty}^4}.$$
It is immediate that, for all $n\geq 0,$ 
\begin{equation}
\label{ChanMod3}
a_2(3n+2)\equiv0\pmod3,
\end{equation}
 a result similar in nature to \eqref{RamCong}. Since then, many authors have studied similar congruences for $a_2(n)$ (see \cite{ChernDast}, \cite{ChuZhou} and references therein).  In particular, we highlight the following result of Chan and Toh \cite{ChanToh} which was proven using modular forms:  
\begin{theorem} 
\label{ChanTohMod5}
For all $n\geq 0,$ 
$$
a_2(5^jn+d_j) \equiv 0 \pmod{5^{\lfloor j/2\rfloor } }
$$
where $d_j$ is the inverse of 8 modulo $5^j$.  
\end{theorem}
%

By way of generalization, we define a \emph{generalized cubic partition} of weight $n$ to be a partition of $n\geq 1$ wherein each even part may appear in $c\geq 1$ different colors.  We denote the number of such generalized cubic partitions by $a_c(n)$ and define $a_c(0):=1$.  Notice that, for all $n\geq 0$, $a_1(n)=p(n)$ while $a_2(n)$ enumerates the cubic partitions of $n$ as described above.
It is clear that the generating function for $a_c(n)$ is given by 
\begin{align} \label{genfn_ac} 
F_c(q):=\sum_{n\geq0}a_c(n)\,q^n&=\prod_{j\geq1}\frac1{(1-q^j)(1-q^{2j})^{c-1}} = \frac{1}{f_1f_2^{c-1}}.
\end{align}

Our primary goal in this work is to prove several congruences modulo primes satisfied by $a_c(n)$ for an infinite family of values of $c$.  The proof techniques that we will utilize will include elementary approaches as well as modular forms.  In particular, we will prove the following infinite family of congruences modulo a prime $p\geq 3$ which are reminiscent of \eqref{ChanMod3}:  

\begin{theorem}
\label{a_p-1_mod_p}
Let $p$ be an odd prime.  Then, for all $n\geq 0$,   
$$a_{p-1}(pn+r) \equiv 0 \pmod{p}$$ 
where $r$ is an integer, $1\leq r\leq p-1$, such that $8r+1$ is a quadratic nonresidue modulo $p$.  
\end{theorem}

We will then show that Theorem \ref{a_p-1_mod_p} can be generalized in the following natural way:  

\begin{cor} 
\label{a_kp-1_mod_p} 
Let $p\geq 3$ be prime and let $r$, $1\leq r\leq p-1$ such that, for all $n\geq 0,$ $a_{p-1}(pn+r) \equiv 0 \pmod{p}$.  Then, for any $k\geq 1$, and for all $n\geq 0,$ $a_{kp-1}(pn+r) \equiv 0 \pmod{p}$.  
\end{cor}

We will also prove two ``isolated'' congruences, modulo the primes 7 and 11, respectively.  
\begin{theorem}
\label{isolated_congs} 
For all $n\geq 0,$ 
\begin{align*}
a_3(7n+4) &\equiv 0 \pmod{7} \textrm{\ \ and} \\
a_5(11n+10) &\equiv 0 \pmod{11}.
\end{align*}
\end{theorem}
These two results will be proved via modular forms (using arguments that are very similar to one another).  
  
The remainder of this paper is organized as follows. 
In Section~\ref{elem_proofs} we present elementary proofs of Theorem~\ref{a_p-1_mod_p} and Corollary~\ref{a_kp-1_mod_p}. 
Section~\ref{modular_basics} contains the basic results from the theory of modular forms which we need to prove Theorem~\ref{isolated_congs}.  We close that section with the proof of Theorem~\ref{isolated_congs}.
Finally, in Section~\ref{summary}, we consider \emph{generalized cubic overpartitions} and prove an infinite family of congruences for the related partition functions via elementary methods.  
We conclude Section~\ref{summary} with some relevant remarks.

\section{Elementary Proofs of Theorem~\ref{a_p-1_mod_p} and Corollary~\ref{a_kp-1_mod_p}}
\label{elem_proofs}

Sellers~\cite[Theorem 2.1]{Sellers} developed, among several other results, a functional equation for $F_2(q)$. That result can be generalized to our family of generating functions.  Namely, thanks to \eqref{genfn_ac}, we can easily show that 
\begin{equation}
\label{functional_equation}
F_c(q) = \psi(q)\psi(q^2)^{c-1}F_c(q^2)^2
\end{equation}
where 
$$
\psi(q): = \frac{f_2^2}{f_1} = \sum_{k\geq 0}q^{k(k+1)/2}
$$ 
is one of Ramanujan's well--known theta functions \cite[p. 6]{Berndt2}.  We can then iterate (\ref{functional_equation}) to prove the following:
\begin{lemma}
\label{genfn1}
Let $p\geq 3$ be prime.  Then 
$$
F_{p-1}(q) = \psi(q)\prod_{i\geq 1}\psi(q^{2^i})^{p\cdot 2^{i-1}}.
$$  
\end{lemma}

\begin{proof}  
We have
\begin{align*}
F_{p-1}(q) 
&= 
\frac{1}{f_1f_2^{p-2}} = \psi(q)\psi(q^2)^{p-2}F_{p-1}(q^2)^2 \\
&= 
\psi(q)\psi(q^2)^{p-2}\left( \psi(q^2)\psi(q^4)^{p-2}F_{p-1}(q^4)^2 \right)^2\\
&= 
\psi(q)\psi(q^2)^{p} \psi(q^4)^{2(p-2)}F_{p-1}(q^4)^4 \\
&= 
\psi(q)\psi(q^2)^{p} \psi(q^4)^{2(p-2)}\left(  \psi(q^4)\psi(q^8)^{p-2} F_{p-1}(q^8)^2 \right)^4 \\
&= 
\psi(q)\psi(q^2)^{p} \psi(q^4)^{2p} \psi(q^8)^{4(p-2)} F_{p-1}(q^8)^8 \\
& \qquad \qquad \qquad \qquad \qquad \qquad  \vdots
\end{align*}
The result follows by continuing to iterate (\ref{functional_equation}) indefinitely.  
\end{proof}

\noindent
Lemma~\ref{genfn1} then leads to a straightforward proof of Theorem~\ref{a_p-1_mod_p}.  

\begin{proof}[Proof of Theorem~\ref{a_p-1_mod_p}] Thanks to Lemma~\ref{genfn1}, we know 
$$F_{p-1}(q) = \psi(q)\prod_{i\geq 1}\psi(q^{2^i})^{p\cdot 2^{i-1}} \equiv \psi(q)\prod_{i\geq 1}\psi(q^{p\cdot 2^i})^{2^{i-1}} \pmod{p}.$$ 
Thus, if we wish to focus our attention on $a_{p-1}(pn+r)$ with the conditions as stated in the theorem, then we only need to consider $\psi(q)$ because 
$$
\prod_{i\geq 1}\psi(q^{p\cdot 2^i})^{2^{i-1}}
$$ 
is a function of $q^p.$  Thus, we simply need to confirm that we have no solutions to the equation 
$$
pn+r = \frac{k(k+1)}{2} 
$$ 
or 
$$
8(pn+r)+1 = 8\left( \frac{k(k+1)}{2} \right) +1  = (2k+1)^2.
$$  
If we did have a solution to the above, then 
$$
8r+1 \equiv (2k+1)^2 \pmod{p},
$$ 
but we have explicitly assumed that $r$ has been selected so that $8r+1$ is a quadratic nonresidue modulo $p$.  Therefore, there can be no such solutions, and this proves our result.  
\end{proof}

\begin{proof}[Proof of Corollary~\ref{a_kp-1_mod_p}]  The generating function for $a_{kp-1}(n)$ is 
\begin{align*}
\frac{1}{f_1f_2^{kp-2}} &= \frac{1}{f_1f_2^{p-2}f_2^{(k-1)p}} \equiv \frac{1}{f_1f_2^{p-2}f_{2p}^{k-1}}  \pmod{p}
\end{align*}
If we wish to focus our attention on $a_{kp-1}(pn+r)$ with the conditions as stated in the corollary, then we only need to look at $\frac{1}{f_1f_2^{p-2}}$  because 
$f_{2p}^{k-1}$ is a function of $q^{p}$.  
Notice that $\frac{1}{f_1f_2^{p-2}}$ is the generating function for $a_{p-1}(n).$  Given that a congruence mod $p$ is assumed to hold for $a_{p-1}(n)$, it must be the case that a similar congruence modulo $p$ will hold for $a_{kp-1}(n)$ on the same arithmetic progression.  
\end{proof}

\section{A Modular Forms Proof of Theorem~\ref{isolated_congs}} 
\label{modular_basics}

\noindent
We begin this section with some definitions and basic facts on modular forms that are instrumental in furnishing our proof of Theorem~\ref{isolated_congs}. For additional details, see for example \cite{koblitz1993, ono2004}. 
We first identify the matrix groups 
\begin{align*}
\text{SL}_2(\mathbb{Z}) & :=\left\{\begin{bmatrix}
a  &  b \\
c  &  d      
\end{bmatrix}: a, b, c, d \in \mathbb{Z}, ad-bc=1
\right\},\\
\Gamma_{0}(N) & :=\left\{
\begin{bmatrix}
a  &  b \\
c  &  d      
\end{bmatrix} \in \text{SL}_2(\mathbb{Z}) : c\equiv 0\pmod N \right\},
\end{align*}
\begin{align*}
\Gamma_{1}(N) & :=\left\{
\begin{bmatrix}
a  &  b \\
c  &  d      
\end{bmatrix} \in \Gamma_0(N) : a\equiv d\equiv 1\pmod N \right\},
\end{align*}
and 
\begin{align*}\Gamma(N) & :=\left\{
\begin{bmatrix}
a  &  b \\
c  &  d      
\end{bmatrix} \in \text{SL}_2(\mathbb{Z}) : a\equiv d\equiv 1\pmod N, ~\text{and}~ b\equiv c\equiv 0\pmod N\right\},
\end{align*}
where $N$ is a positive integer. A subgroup $\Gamma$ of the group $\text{SL}_2(\mathbb{Z})$ is called a \emph{congruence subgroup} if $\Gamma(N)\subseteq \Gamma$ for some $N$. The smallest $N$ such that $\Gamma(N)\subseteq \Gamma$
is called the \emph{level of $\Gamma$}. For example, $\Gamma_0(N)$ and $\Gamma_1(N)$
are congruence subgroups of level $N$. 

\smallskip
\noindent
Let $\mathbb{H}:=\{z\in \mathbb{C}: \text{Im}(z)>0\}$ be the upper half of the complex plane. Then, the following subgroup of the general linear group
$$\text{GL}_2^{+}(\mathbb{R})=\left\{\begin{bmatrix}
a  &  b \\
c  &  d      
\end{bmatrix}: a, b, c, d\in \mathbb{R}~\text{and}~ad-bc>0\right\}$$ acts on $\mathbb{H}$ by $\begin{bmatrix}
a  &  b \\
c  &  d      
\end{bmatrix} z=\displaystyle \frac{az+b}{cz+d}$.  
We identify $\infty$ with $\displaystyle\frac{1}{0}$ and define 
$$\begin{bmatrix}
a  &  b \\
c  &  d      
\end{bmatrix} \displaystyle\frac{r}{s}=\displaystyle \frac{ar+bs}{cr+ds},$$ 
where $\displaystyle\frac{r}{s}\in \mathbb{Q}\cup\{\infty\}$.
This gives an action of $\text{GL}_2^{+}(\mathbb{R})$ on the extended upper half-plane $\mathbb{H}^{\ast}=\mathbb{H}\cup\mathbb{Q}\cup\{\infty\}$. 
Suppose that $\Gamma$ is a congruence subgroup of $\text{SL}_2(\mathbb{Z})$. A \emph{cusp} of $\Gamma$ is an equivalence class in $\mathbb{P}^1=\mathbb{Q}\cup\{\infty\}$ under the action of $\Gamma$.
The group $\text{GL}_2^{+}(\mathbb{R})$ also acts on functions $f: \mathbb{H}\rightarrow \mathbb{C}$. In particular, suppose that $\gamma=\begin{bmatrix}
a  &  b \\
c  &  d      
\end{bmatrix}\in \text{GL}_2^{+}(\mathbb{R})$. If $f(z)$ is a meromorphic function on $\mathbb{H}$ and $\ell$ is an integer, then define the slash operator $|_{\ell}$ by 
$$(f|_{\ell}\gamma)(z):=(\text{det}~{\gamma})^{\ell/2}(cz+d)^{-\ell}f(\gamma z).$$
\begin{definition}

\noindent
Let $\Gamma$ be a congruence subgroup of level $N$. A holomorphic function $f: \mathbb{H}\rightarrow \mathbb{C}$ is called a 
\emph{modular form} with integer weight $\ell$ on $\Gamma$ if the following hold:
	\begin{enumerate}
		\item We have $f\left(\displaystyle \gamma z\right)=(cz+d)^{\ell}f(z)$ for all $z\in \mathbb{H}$ and all $\gamma=\begin{bmatrix}
		a  &  b \\
		c  &  d      
		\end{bmatrix} \in \Gamma$.
		\item If $\gamma\in \text{SL}_2(\mathbb{Z})$, then $(f|_{\ell}\gamma)(z)$ has a Fourier expansion of the form $$(f|_{\ell}\gamma)(z)=\displaystyle\sum_{n\geq 0}a_{\gamma}(n)q_N^n,$$
		where $q_N:=e^{2\pi iz/N}$.
	\end{enumerate}
\end{definition}

\noindent
For a positive integer $\ell$, the complex vector space of modular forms of weight $\ell$ with respect to a congruence subgroup $\Gamma$ is denoted by $M_{\ell}(\Gamma)$.
\begin{definition}\cite[Definition 1.15]{ono2004}
	If $\chi$ is a Dirichlet character modulo $N$, then we say that a modular form $f\in M_{\ell}(\Gamma_1(N))$ has 
\emph{Nebentypus character} $\chi$ if
	$$f\left( \frac{az+b}{cz+d}\right)=\chi(d)(cz+d)^{\ell}f(z)$$ for all $z\in \mathbb{H}$ and all $\begin{bmatrix}
	a  &  b \\
	c  &  d      
	\end{bmatrix} \in \Gamma_0(N)$. The space of such modular forms is denoted by $M_{\ell}(\Gamma_0(N), \chi)$. 
\end{definition}

\noindent
In this paper, the relevant modular forms are those that arise from eta-quotients. The \emph{Dedekind eta-function} $\eta(z)$ is defined by
\begin{align*}
\eta(z):=q^{1/24}(q;q)_{\infty}=q^{1/24}\prod_{n=1}^{\infty}(1-q^n),
\end{align*}
where $q:=e^{2\pi iz}$ and $z\in \mathbb{H}$, the upper half-plane. A function $f(z)$ is called an \emph{eta-quotient} if it is of the form
\begin{align*}
f(z)=\prod_{\delta\mid N}\eta(\delta z)^{r_\delta},
\end{align*}
where $N$ is a positive integer and $r_{\delta}$ is an integer. We now recall two valuable theorems from \cite[p. 18]{ono2004} which will be used to prove our results.
\begin{theorem}\cite[Theorem 1.64]{ono2004}\label{thm_ono1} If $f(z)=\prod_{\delta\mid N}\eta(\delta z)^{r_\delta}$ 
	is an eta-quotient such that $\ell=\frac{1}{2}\sum_{\delta\mid N}r_{\delta}\in \mathbb{Z}$, 
	$$\sum_{\delta\mid N} \delta r_{\delta}\equiv 0 \pmod{24}\ \ \ \ \ \ \text{and}\ \ \ \ \ \
	\sum_{\delta\mid N} \frac{N}{\delta}r_{\delta}\equiv 0 \pmod{24},$$
	then $f(z)$ satisfies $f\left( \gamma z\right)=\chi(d)(cz+d)^{\ell}f(z)$
	for every  $\gamma=\begin{bmatrix}
	a  &  b \\
	c  &  d      
	\end{bmatrix} \in \Gamma_0(N)$. 
Here the character $\chi$ is defined by $\chi(\bullet):=\left(\frac{(-1)^{\ell} s}{\bullet}\right)$, where $s:= \prod_{\delta\mid N}\delta^{r_{\delta}}$. 
\end{theorem}

\noindent
Suppose that $f$ is an eta-quotient satisfying the conditions of Theorem \ref{thm_ono1} and that the associated weight $\ell$ is a positive integer. If the function $f(z)$ is holomorphic at all of the cusps of $\Gamma_0(N)$, then $f(z)\in M_{\ell}(\Gamma_0(N), \chi)$. The next theorem gives a necessary condition for determining orders of an eta-quotient at cusps.
\begin{theorem}\cite[Theorem 1.65]{ono2004}\label{thm_ono2}
	Let $c, d$ and $N$ be positive integers with $d\mid N$ and $\gcd(c, d)=1$. If $f$ is an eta-quotient satisfying the conditions of Theorem \ref{thm_ono1} for $N$, then the 
	order of vanishing of $f(z)$ at the cusp $\frac{c}{d}$ 
	is $$\frac{N}{24}\sum_{\delta\mid N}\frac{\gcd(d,\delta)^2r_{\delta}}{\gcd(d,\frac{N}{d})d\delta}.$$
\end{theorem}

\noindent
We now remind ourselves a result of Sturm \cite{Sturm1984} which gives a criterion to test whether two modular forms are congruent modulo a given prime.
\begin{theorem}\label{Sturm}
	Let $k$ be an integer and $g(z)=\sum_{n=0}^\infty a(n)q^n$ a modular form of weight $k$ for $\Gamma_{0}(N)$. For any given positive integer $m$, if $a(n)\equiv 0 \pmod{m}$ holds for all 
$\displaystyle{n\leq \frac{kN}{12} \prod_{p~ \text{prime} \atop p\,|\,N}~\left(1+\frac1p\right)}\, ,$
 then $a(n)\equiv 0 \pmod m$ holds for any $n\geq 0$.
\end{theorem}

\noindent
We now recall the definition of Hecke operators.
Let $m$ be a positive integer and $f(z) = \sum_{n=0}^{\infty} a(n)q^n \in M_{\ell}(\Gamma_0(N),\chi)$. Then the action of Hecke operator $T_m$ on $f(z)$ is defined by 
\begin{align*}
f(z)\,|\,T_m := \sum_{n=0}^{\infty} \left(\sum_{d\mid \gcd(n,m)}\chi(d)d^{\ell-1}a\left(\frac{nm}{d^2}\right)\right)q^n.
\end{align*}
In particular, if $m=p$ is prime, we have 
\begin{align}\label{Tp}
f(z)\,|\,T_p := \sum_{n=0}^{\infty} \left(a(pn)+\chi(p)p^{\ell-1}a\left(\frac{n}{p}\right)\right)q^n.
\end{align}
We take by convention that $a(n/p)=0$ whenever $p \nmid n$. The next result  follows directly from \eqref{Tp}.

\begin{proposition} \label{prophecke}
Let $p$ be a prime, $g(z)\in \Z[[q]], h(z)\in\Z[[q^p]]$, and $k > 1$. Then, we have
\begin{align*}
\left(g(z)h(z)\right)\,|\,T_p\equiv\left(g(z)|T_p\cdot h(z/p)\right)\pmod{p}.
\end{align*}	
\end{proposition}

With the above in place, we are ready to supply the proof of Theorem~\ref{isolated_congs}.

\begin{proof}[Proof of Theorem~\ref{isolated_congs}]
We begin by proving the mod 7 congruence that appears in the theorem.  By taking $c=3$ in \eqref{genfn_ac}, we have
\begin{align} \label{thm1.1(III)} \sum_{n\geq0}a_3(n)\,q^n&=\prod_{k\geq1}\frac1{(1-q^k)(1-q^{2k})^{2}}.
\end{align}

\noindent
Let $$H(z):= \frac{\eta^{76}(z)}{\eta^2(2z)}.$$

\noindent
By Theorems \ref{thm_ono1} and \ref{thm_ono2}, we find that $H(z)$ is a modular form of weight $37$, level $8$ and character $\chi_1=(\frac{-2^{-2}}{\bullet})$. 
By \eqref{thm1.1(III)}, the Fourier expansion of our form satisfies
\begin{align*}
H(z)=\left(\sum_{n=0}^{\infty}a_3(n)q^{n+3}\right)\prod_{k\geq1}(1-q^k)^{77}.
\end{align*}

\noindent
Using Proposition \ref{prophecke}, we calculate that 
\begin{align*}
H(z)\,|\,T_{7}\equiv \left(\sum_{n=0}^{\infty} a_3(7n+4)q^{n+1} \right)\prod_{k\geq1}(1-q^k)^{11}\pmod{7}.
\end{align*}
Since the Hecke operator is an endomorphism on $ M_{37}\left(\Gamma_{0}(8), \chi_1\right)$, we have that $H(z)\,|\,T_{7}\in M_{37}\left(\Gamma_{0}(8), \chi_1\right)$.  By Theorem \ref{Sturm}, the Sturm bound for this space of forms is $37$. Using $SageMath$, we verify that the Fourier coefficients of $H(z)\,|\,T_{7}$ up to the desired bound are congruent  to 0  modulo 7. Hence, Theorem \ref{Sturm} confirms that $H(z)\,|\,T_{7}\equiv 0\pmod{7}$. This completes the proof of the mod 7 congruence.

We next consider the mod 11 congruence in the theorem whose proof is rather similar to the proof of the mod 7 congruence above.  
By \eqref{genfn_ac}, we have
\begin{align} \label{thm1.1} \sum_{n\geq0}a_5(n)\,q^n&=\prod_{k\geq1}\frac1{(1-q^k)(1-q^{2k})^{4}}.
\end{align}

\noindent
Let $$G(z):= \frac{\eta^{32}(z)}{\eta^4(2z)}.$$

\noindent
By Theorems \ref{thm_ono1} and \ref{thm_ono2}, we find that $G(z)$ is a modular form of weight $14$, level $4$ and character $\chi_0=(\frac{2^{-4}}{\bullet})$. 
By \eqref{thm1.1}, the Fourier expansion of our form satisfies
\begin{align*}
G(z)=\left(\sum_{n=0}^{\infty}a_5(n)q^{n+1}\right)\prod_{k\geq1}(1-q^k)^{33}.
\end{align*}

\noindent
Using Proposition \ref{prophecke}, we calculate that 
\begin{align*}
G(z)\,|\,T_{11}\equiv \left(\sum_{n=0}^{\infty} a_5(11n+10)q^{n+1} \right)\prod_{k\geq1}(1-q^k)^{3}\pmod{11}.
\end{align*}
Since the Hecke operator is an endomorphism on $ M_{14}\left(\Gamma_{0}(4), \chi_0\right)$, we have that $G(z)\,|\,T_{11}\in M_{14}\left(\Gamma_{0}(4), \chi_0\right)$.  By Theorem \ref{Sturm}, the Sturm bound for this space of forms is $7$. Hence, Theorem \ref{Sturm} confirms that $G(z)\,|\,T_{11}\equiv 0\pmod{11}$. This completes the proof of the mod 11 congruence. 
\end{proof}

\section{Closing Thoughts} \label{summary}

We now transition to a consideration of combinatorial objects related to generalized cubic partitions, namely, generalized cubic overpartitions.    

An \emph{overpartition} of $n$ is a partition of $n$ in which the first occurrence of a part may be overlined.  
In 2010, Kim \cite{Kim} introduced an overpartition version of cubic partitions, sometimes called \emph{overcubic partitions}. Following the four cubic partitions of $n=3$ from the Introduction, we find that the corresponding $12$ overcubic partitions of $n=3$ are given by the following:
\begin{align*} &3, \ \ \overline{3}, \ \ 2 + 1, \ \ {\color{red}2} + 1, \ \ \overline{2} + 1, \ \ {\color{red}\overline{2}} + 1, \\
&2 + \overline{1}, \ \ {\color{red}2} + \overline{1}, \ \ \overline{2} + \overline{1}, \ \ {\color{red}\overline{2}} + \overline{1}, \ \ 1 + 1 + 1, \ \ \overline{1} + 1 +  1.
\end{align*}
If we now generalize these to mirror our generalization of cubic partitions, then the generating function for these cubic overpartitions is given by 
$$\overline{F}_c(q):=\sum_{n\geq0} \overline{a}_c(n)\,q^n
=\frac{ (-q;q)_\infty (-q^2; q^2)^{c-1}_\infty }{ (q;q)_{\infty} (q^2;q^2)_{\infty}^{c-1} }.$$

\noindent
We now prove the following natural companion to Corollary~\ref{a_kp-1_mod_p}.

\begin{theorem}  \label{James2} Let $p\geq 3$ be prime, and let $r$, $1\leq r\leq p-1$, such that r is a quadratic nonresidue modulo $p$.  Also let $k$ be a positive integer.  
Then, for all $n\geq 0,$  $\overline{a}_{kp-1}(pn+r) \equiv 0 \pmod{p}.$  
\end{theorem}

\begin{proof} After elementary manipulations, the generating function for $\overline{a}_c(n)$ can be given in the form
$$\overline{F}_c(q)=\frac{ f_4^{c-1} }{ f_1^2f_2^{2c-3} }. $$
Now consider the above generating function for $c = kp-1$.  Then we have 

\begin{align*}
\sum_{n\geq 0}  \overline{a}_{kp-1}(n) q^n 
&= \frac{ f_4^{kp-2} }{ f_1^2f_2^{2kp-5} } 
= \left(  \frac{f_4^{kp}}{f_2^{2kp}} \right) \left( \frac{ f_2^5   }{ f_1^2f_4^2 }  \right) 
=  \left(  \frac{f_4^{kp}}{f_2^{2kp}} \right) \cdot \varphi(q)  \\
& \equiv \left(  \frac{f_{4p}^k}{f_{2p}^{2k}}  \right) \cdot \varphi(q)  \pmod{p}
\end{align*}
where 
$$\varphi(q): = 1+2\sum_{j\geq 1} q^{j^2}$$
is another of Ramanujan's well--known theta functions \cite[p. 6]{Berndt2}. 
Note that 
$$
\frac{f_{4p}^k}{f_{2p}^{2k}}
$$
is a function of $q^p$, which means that, in order to prove this theorem, we simply need to consider whether $pn+r = j^2$ for some $r$ and $j$.  Since $r$ is assumed to be a quadratic nonresidue modulo $p$, we know that there can be no solutions to the equation  $pn+r = j^2$.  This proves our result.  
\end{proof}

We close this paper with two sets of brief remarks.  

\begin{enumerate}

\item{} 
We first return to the generalized cubic partitions and make the following elementary observation.  For all $c>1$ which satisfy $c\equiv 1 \pmod p$ for $p=5, 7$ or $11$, it is clear that $a_c(n)$ ``inherits'' the corresponding congruence modulo $5$, $7$, or $11$ as Ramanujan's \eqref{RamCong}.  That is to say, for all $j\geq 0$ and $n\geq 0,$ 
\begin{align*} 
a_{5j+1}(5n+4)&\equiv0\pmod 5,  \\  
a_{7j+1}(7n+5)& \equiv0\pmod 7, \textrm{\ \ and} \\
a_{11j+1}(11n+6)& \equiv0\pmod{11}. 
\end{align*}

\item{} 
As we mentioned above, our emphasis in this paper has been on congruences modulo a prime $p$ satisfied by $a_c(n)$ or $\overline{a}_c(n)$ for particular values of $c$.  We encourage the interested reader to consider identifying and proving additional congruences for these families of functions modulo powers of a prime akin to Theorem~\ref{ChanTohMod5}. 

\end{enumerate}

\end{document}